\documentclass{amsart}
\usepackage{amssymb,graphpap}
\usepackage{euler}
\usepackage{enumitem}
\newtheorem{theorem}{Theorem}[section]

\newtheorem{corollary}[theorem]{Corollary}

\theoremstyle{definition}

\theoremstyle{remark}

\newcommand{\U}{\mathbb{U}}
\newcommand{\C}{\mathbb{C}}
\newcommand{\Wg}{\operatorname{Wg}}
\newcommand{\Moeb}{\operatorname{Moeb}}

\newcommand{\Tr}{\operatorname{Tr}}

\newcommand{\CC}{\mathcal{C}}
\newcommand{\Cat}{\operatorname{Cat}}
\newcommand{\perm}{\operatorname{perm}}
\newcommand{\Sym}{\operatorname{Sym}}

\begin{document}
\date{\today}
\sloppy

\title[JM Elements and the Weingarten Function]
{Complete Symmetric Polynomials in Jucys-Murphy Elements and the Weingarten Function}

\author[J. I. Novak]{Jonathan Novak}
\address{Queen's University, Department of Mathematics and Statistics,
Jeffery Hall, Kingston, ON K7L 3N6, Canada}
\email{jnovak@mast.queensu.ca}
\urladdr{www.mast.queensu.ca/jnovak}
	
	\begin{abstract}
	A connection is made between complete homogeneous symmetric polynomials 
	in Jucys-Murphy elements and the unitary Weingarten function from random matrix
	theory.  In particular we show that $h_r(J_1,\dots,J_n),$ the complete homogeneous symmetric
	polynomial of degree $r$ in the JM elements, coincides with the $r$th term in the 
	asymptotic expansion of
	the Weingarten function.  We use this connection to determine precisely which
	conjugacy classes occur in the class basis resolution of $h_r(J_1,\dots,J_n),$ and
	to explicitly determine the coefficients of the classes of minimal height when $r < n.$
	These coefficients, which turn out to be products of Catalan numbers, are governed by the 
	Moebius function of the non-crossing partition lattice $NC(n).$
	\end{abstract}

	\maketitle
	
	\section{Introduction}
   	Consider the tower
	\begin{equation} 
		\{1\} = \C[S(1)] \subset \C[S(2)] \subset \dots \C[S(n)] \subset \dots
	\end{equation}
	of the group algebras of the symmetric groups, where $\C[S(n-1)]$ is canonically
	embedded in $\C[S(n)]$ as the linear span of permutations having $n$ as a fixed
	point.  Recently, Okounkov and Vershik \cite{OV} have developed a novel approach to 
	the representation theory of the symmetric groups in which a key role is played the
	the {\it Gelfand-Zetlin} algebra $GZ(n),$ which is by definition the algebra generated
	by $Z(1),\dots,Z(n),$ where $Z(i)$ is the center of $\C[S(i)].$
	
	Okounkov and Vershik give a remarkably concrete description of the 
	Gelfand-Zetlin algebra: they prove that $GZ(n)=\C[J_1,\dots,J_n],$ where
	the $J_i$'s are special elements of $\C[S(n)]$ called the {\it Jucys-Murphy} elements.
	The JM elements are defined by $J_1:=0$ and
	\begin{equation}
		J_i= (1,i) + \dots + (i-1,i)
	\end{equation}
	for $2 \leq i \leq n.$  For example, the JM elements in $\C[S(4)]$ are
	\begin{align*}
		J_1 &= 0 \\
		J_2 &= (1,2) \\
		J_3 &= (1,3) + (2,3) \\
		J_4 &= (1,4) + (2,4) + (3,4).
	\end{align*}
	
	Given Okounkov and Vershik's characterization of $GZ(n),$ it is natural to ask 
	for a characterization of $Z(n)$ itself in terms of the JM elements.  It was shown
	by Jucys \cite{J} in $1974$ that $Z(n)=\Sym[J_1,\dots,J_n],$ the algebra of 
	{\it symmetric} polynomials in the JM elements.  Jucys proved this remarkable
	fact by explicitly evaluating the elementary symmetric polynomials in JM elements.
	
	For each partition (or Young diagram) $\mu \vdash n,$ let $\CC_{\mu}$ denote the 
	formal sum of all permutations in $S(n)$ of cycle type $\mu.$  Then 
	$\{\CC_{\mu}\}_{\mu \vdash n}$ is a natural basis of $Z(n),$ called the 
	{\it class basis}.  If one thinks of $Z(n)$ as the algebra of functions $f:S(n)
	\rightarrow \C$ which are constant on conjugacy classes, then $\CC_{\mu}$
	is the indicator function of the conjugacy class labelled by $\mu.$
			
	\begin{theorem}[\cite{J}]
	\label{Jucys}
	For $1 \leq r \leq n,$ let 
	\begin{equation}
		e_r(J_1,\dots,J_n) = \sum_{1 \leq i_1 < i_2 < \dots < i_r \leq n} 
		J_{i_1}J_{i_2} \dots J_{i_n}
	\end{equation}
	be the $r$th elementary symmetric polynomial in the JM elements, and put
	$e_0:=1$ (the identity element of $\C[S(n)]$).  Then for $0 \leq r \leq n,$
	\begin{equation}
		e_r(J_1,\dots,J_n) = \sum_{\substack{\mu \vdash n\\ \ell(\mu) = n-r}}
		\CC_{\mu}.
	\end{equation}
	\end{theorem}
	\noindent
	In other words, Jucys' result states that $e_r(J_1,\dots,J_n)$ is the indicator function
	of the set of permutations with exactly $n-r$ cycles.  For example, the 
	class basis resolution of $e_2(J_1,J_2,J_2,J_4)$ is the sum of all classes
	indexed by Young diagrams
	$\mu \vdash 4$ with exactly $2$ rows:
	\begin{equation}
		e_2(J_1,J_2,J_3,J_4) = \CC_{(3,1)} + \CC_{(2,2)}.
	\end{equation}
	
	In symmetric function theory, the complete homogeneous symmetric polynomials
	are in many ways ``dual'' to the elementary symmetric polynomials.  Recall
	that $h_r \in \Sym[x_1,\dots,x_n]$ are defined by $h_0:=1$ and
	\begin{equation}
		h_r := \sum_{1 \leq i_1 \leq i_2 \leq \dots \leq i_r \leq n}x_{i_1}x_{i_2} \dots
		x_{i_r}
	\end{equation}
	for $r \geq 1.$  Clearly $e_r =0$ for $r > n,$ but this is not the case for the 
	$h_r$'s.  Indeed one has the generating functions
	\begin{equation}
		E(t) = \sum_{r \geq 0} e_r t^r = (1+x_1t)(1+x_2t) \dots (1+x_nt),
	\end{equation}
	valid as an identity in $\C[x_1,\dots,x_n,t],$ and 
	\begin{equation}
		H(t) = \sum_{r \geq 0} h_rt^r = (1-x_1t)^{-1}(1-x_2t)^{-1} \dots (1-x_nt)^{-1},
	\end{equation}
	valid as an identity in $\C[x_1,\dots,x_n][[t]].$  In particular, one has the identity
	$E(t)H(-t)=1$ in the latter ring.
	
	Another example of the duality between $e_r$ and $h_r$ can be expressed in 
	terms of antisymmetric and symmetric powers of matrices.  If $A$ is an $n \times n$
	matrix with eigenvalues $x_1,\dots,x_n,$ then
	\begin{equation}
		e_r(x_1,\dots,x_n) = \Tr \wedge^r A
	\end{equation}
	while
	\begin{equation}
		h_r(x_1,\dots,x_n)= \Tr \vee^r A,
	\end{equation}
	see e.g. \cite{B}.  
	
	Given the dual nature of the elementary and complete homogeneous symmetric polynomials,
	one would hope for an analogue of Jucys result (Theorem \ref{Jucys}) for
	$h_r(J_1,\dots,J_n).$  Below we will prove that, for any $r \geq 0,$
	\begin{equation}
		h_r(J_1,\dots,J_n) = (-1)^r \sum_{\mu \vdash n} a_r(\mu) \CC_{\mu},
	\end{equation}
	where $a_r$ is precisely the $r$th term in the asymptotic expansion of the {\it
	Weingarten function}, a special function which occurs in the theory of random unitary 
	matrices.  
	
	Using known results about the Weingarten function, we characterize
	the conjugacy classes which appear in the class basis resolution of $h_r(J_1,\dots,J_n),$
	and give exact formulas for some of the coefficients.  The coefficients which we are
	able to determine explicitly correspond to ``minimal classes,'' and
	turn out to be governed by the Moebius function of the 
	non-crossing partition lattice $NC(n).$
	
	\section{The Weingarten Function}
	
	\subsection{Weingarten function in random matrix theory}
	A random matrix may be described either ``locally,'' by giving the joint distribution of 
	the matrix entries, or ``globally''' as a probability distribution on matrix space.  For
	example, the Ginibre Unitary Ensemble is described locally as the class of random matrices
	whose entries are i.i.d. standard complex Gaussian random variables, while the 
	Circular Unitary Ensemble is described globally in terms of the Haar probability measures 
	on the unitary groups.  In the latter case, a non-trivial problem is to recover the 
	local description from the global one.  This consists of determining the joint moments
	of the matrix entries of a Haar-distributed random unitary matrix, 
	which is equivalent to evaluating all matrix integrals of the form
	\begin{equation}
		I_d(i,j,i',j')=\int_{\U_d} u_{i(1)j(1)} \dots u_{i(n)j(n)} 
		\overline{u}_{i'(1)j'(1)} \dots \overline{u}_{i'(n)j'(n)}dU,
	\end{equation}
	where 
	\begin{equation}
		\U_d=\{U \in GL_d(\C) : U^*=U^{-1}\}
	\end{equation}
	is the unitary group equipped with normalized Haar measure $dU,$ and 
	$i,j,i',j':\{1,\dots,n\} \rightarrow \{1,\dots,d\}$ are arbitrary functions.  Of particular 
	interest is the large $d$ limit of the joint moments.
		
	Considerable progress on this problem has been made by Collins \cite{C} and
	Collins-\'Sniady \cite{CS}, who make use of the classical Schur-Weyl duality
	between representations of $\U_d$ and $S(n)$ in the tensors
	\begin{equation}
		\underbrace{\C^d \otimes \dots \otimes \C^d}_n,
	\end{equation}
	to exhibit the existence of a central function $\Wg_d \in Z(n)$ with the following
	remarkable property:
		\begin{equation}
		\label{convolution}
		I_d(i,j,i',j')=\sum_{\sigma, \tau \in S(n)} 
		\prod_{k=1}^n \delta_{i(k)i'(\sigma(k))} \delta_{j(k)j'(\tau(k))} 
		\Wg_d(\tau \sigma^{-1}).
	\end{equation}
	The function $\Wg_d$ is called the {\it Weingarten function}, and 
	identity (\ref{convolution}) is the {\it Weingarten convolution formula}.  It is similar 
	in spirit to the Wick formula for Gaussian random variables.  It follows 
	immediately from (\ref{convolution}) that $\Wg_d$ has the integral 
	representation
	\begin{equation}
		\Wg_d(\sigma) = \int_{\U_d} u_{11} \dots u_{nn}
		\overline{u}_{1\sigma(1)} \dots \overline{u}_{n \sigma(n)} dU,
	\end{equation}
	valid when $d \geq n.$
	In principle,
	the Weingarten convolution formula reduces the computation of the joint moments
	of the matrix entries of a random unitary matrix to the problem of explicitly describing
	the Weingarten function.  
	
	\subsection{Weingarten function and JM elements}
	We will show that $\Wg_d$ is essentially the generating function
	for complete homogeneous symmetric polynomials in JM elements.  Our approach is based
	on the following remarkable result due to Collins \cite{C}.
	
	\begin{theorem}[\cite{C}]
		\label{inverse}
	For any $d \geq n,$ $\Wg_d$ is an invertible element of $Z(n).$  The class basis
	resolution of $\Wg_d^{-1}$ is
		\begin{equation}
		\Wg_d^{-1}=\sum_{\mu \vdash n} d^{\ell(\mu)} \CC_{\mu}.
	\end{equation}
	\end{theorem}
	
	Let $H(t) \in \C[J_1,\dots,J_n][[t]]$ be the generating function for complete 
	homogeneous symmetric polynomials in JM elements:
	\begin{equation}
		H(t) = \sum_{r \geq 0} h_r(J_1,\dots,J_n)t^r
		=(1-J_1t)^{-1}(1-J_2t)^{-1} \dots (1-J_nt)^{-1}.
	\end{equation}
	This is a formal power series in $t$ with coefficients in the Gelfand-Zetlin subalgebra
	of $\C[S(n)]$ (in fact in $Z(n)$).
	
	\begin{theorem}[\cite{N2}]
	\label{main}
	For any $d \geq n,$
	\begin{equation}
		\Wg_d = \frac{1}{d^n}H(-\frac{1}{d})=(d+J_1)^{-1}(d+J_2)^{-1} \dots (d+J_n)^{-1}.
	\end{equation}
	\end{theorem}
	
	\begin{proof}
		Observe that 
		\begin{align*}
		d^n E(\frac{1}{d}) &= d^n \sum_{r=0}^n e_r(J_1,\dots,J_n) d^{-r} \\
		&= \sum_{r =0}^n d^{n-r} \sum_{\substack{\mu \vdash n\\ \ell(\mu)=n-r}} \CC_{\mu} \\
		&= \sum_{\mu \vdash n} d^{\ell(\mu)} \CC_{\mu} \\
		&= \Wg_d^{-1},
		\end{align*}
		where the second line follows by applying Jucys' theorem.  Since
		$E(t)H(-t)=1,$ the result follows.
	\end{proof}
	
	\subsection{Asymptotic expansion of $\Wg_d$ and the Moebius function}
	Consider now the resolution of $\Wg_d$ with
	respect to the class basis of $Z(n):$
	\begin{equation}
		\Wg_d = \sum_{\mu \vdash n} \Wg_d(\mu) \CC_{\mu}.
	\end{equation}
	Collins \cite{C} demonstrated the existence of a sequence of central functions
	\begin{equation}
		\label{LaurentCoefficients}
		a_0,a_1,\dots,a_r, \dots \in Z(n)
	\end{equation}
	with the property that
	\begin{equation}
		\label{asymptotic}
		\Wg_d(\mu) = \frac{a_0(\mu)}{d^n} + \frac{a_1(\mu)}{d^{n+1}}+ \dots + 
		\frac{a_r(\mu)}{d^{n+r}} + \dots
	\end{equation}
	for all $\mu \vdash n.$  The series (\ref{asymptotic}) is the {\it asymptotic expansion} of
	$\Wg_d.$
	
	Collins \cite{C} has shown that the first order asymptotics of
	$\Wg_d$ are governed by the Moebius function of the non-crossing partition
	lattice $NC(n).$  Recall that $NC(n)$ consists of the 
	set of all non-crossing partitions of $\{1,\dots,n\}$ under the reverse refinement
	partial order: $\pi \leq \pi'$ in $NC(n)$ if and only if every block of $\pi$ is contained in 
	a block of $\pi'.$  $NC(n)$ is a lattice with minimal element
	\begin{equation}
		0_n=\{1\} \sqcup \{2\} \sqcup \dots \sqcup \{n\}
	\end{equation}
	and maximal element
	\begin{equation}
		1_n=\{1,2,\dots,n\}.
	\end{equation}
	
	Non-crossing partition lattices were first studied by Kreweras \cite{K} from a 
	purely combinatorial point of view.  Later, Speicher discovered that non-crossing
	partition lattices play the same role in Voiculescu's free probability theory that
	full partition lattices play in classical probability, see \cite{NS} for further information.
	The Moebius function of $NC(n)$ is determined as follows: for 
	$\pi=V_1 \sqcup \dots \sqcup V_{\ell} \in NC(n),$
	\begin{equation}
		\Moeb([0_n,\pi])=(-1)^{n-\ell} \prod_{i=1}^{\ell} \Cat_{|V_i|-1},
	\end{equation}
	where
	\begin{equation}
		\Cat_N:=\frac{1}{N+1}{2N \choose N}
	\end{equation}
	is the Catalan number.  
	
	$NC(n)$ is canonically embedded in the symmetric
	group $S(n)$ by mapping $\pi =V_1 \sqcup \dots \sqcup V_{\ell}$ 
	onto the permutation $\perm_{\pi}$
	with cycle structure determined by $V_1,\dots,V_{\ell}$ in the natural way.
	For example, the partition 
	\begin{equation}
		\pi = \{1,4,5\} \sqcup \{2,3\} \in NC(5)
	\end{equation}
	is identified with the permutation
	\begin{equation}
		\perm_{\pi} = (1,4,5)(2,3) \in S(5).
	\end{equation}
	Under this identification, the Moebius function becomes a central function
	on $S(n)$ whose class basis resolution is
		\begin{equation}
		\Moeb = \sum_{\mu \vdash n} \Moeb(\mu) \CC_{\mu},
	\end{equation}
	where
	\begin{equation}
		\Moeb(\mu)=\Moeb(\mu_1,\dots,\mu_{\ell})=(-1)^{n-\ell}
		\prod_{i=1}^{\ell} \Cat_{\mu_i-1}
	\end{equation}
	for each $\mu \vdash n.$
		
	\begin{theorem}[\cite{C}]
		\label{Collins}
	The coefficients $a_0,a_1,\dots,a_r,\dots$ in the asymptotic expansion of 
	$\Wg_d$ have the following properties:
		\begin{itemize}
		\item
		for any $r \geq 0$ and $\mu \vdash n,$ $a_r(\mu)$ is non-zero if and only if there exists
		an integer $g \geq 0$ such that
			\begin{equation}
				\ell(\mu)=n-r+2g.
			\end{equation}		
		\item
		If $\mu$ satisfies $\ell(\mu)=n-r,$ then
			\begin{equation}
				a_{n-\ell(\mu)}(\mu)=\Moeb(\mu).
			\end{equation}
		
		\end{itemize}
	\end{theorem}
	
	\subsection{The class expansion of $h_r(J_1,\dots,J_n)$}
	According to Theorem \ref{main}
	\begin{equation}
		\Wg_d = \frac{1}{d^n}H(-\frac{1}{d})=\sum_{r \geq 0}(-1)^rh_r(J_1,\dots,J_n)\frac{1}
		{d^{n+r}}.
	\end{equation}
	Since this is precisely the asymptotic expansion of $\Wg_d,$ we have the following
	result.
	
	\begin{theorem}
		\label{equivalence}
		The coefficients in the asymptotic expansion of $\Wg_d$ are the complete homogeneous
		symmetric polynomials in JM elements:
		\begin{equation}
			a_r=(-1)^rh_r(J_1,\dots,J_n).
		\end{equation}
		Equivalently the coefficients in the asymptotic expansion of $\Wg_d$ are the
		coefficients in the class basis resolution of complete homogeneous symmetric 
		polynomials in JM elements:
		\begin{equation}
			h_r(J_1,\dots,J_n)=(-1)^r \sum_{\mu \vdash n} a_r(\mu) \CC_{\mu}.
		\end{equation}
	\end{theorem}
		
	Collins' theorem \ref{Collins} on the asymptotic expansion of $\Wg_d(\mu)$ says
	that the coefficients $a_r(\mu)$ go down in steps of two, starting at $a_{n-\ell(\mu)}(\mu)
	=\Moeb(\mu).$ 
	
	\begin{corollary}
	The class basis resolution of $h_r(J_1,\dots,J_n)$ is of the form
	\begin{equation}
		h_r(J_1,\dots,J_n) = (-1)^r \sum_{g \geq 0} \sum_{\substack{\mu \vdash n\\
		\ell(\mu)=n-r+2g}} a_r(\mu) \CC_{\mu}.
	\end{equation}
	When $n > r \geq 1,$ this is
	\begin{equation}
		h_r(J_1,\dots,J_n) = (-1)^r \sum_{\substack{\mu \vdash n\\ \ell(\mu)=n-r}}
		\Moeb(\mu) \CC_{\mu} + (-1)^r\sum_{g \geq 1} \sum_{\substack{\nu \vdash n\\
		\ell(\nu)=n-r+2g}} a_r(\nu) \CC_{\mu}.
	\end{equation}
	\end{corollary}
	
	As an example, consider the class basis resolution of $h_2(J_1,J_2,J_3,J_4)$ 
	(recall that we computed $e_2(J_1,J_2,J_3,J_4)$ above as an example of Jucys'
	theorem).  Since $4-2=2,$ only classes of height $2$ (corresponding to $g=0$) or
	$4$ (corresponding to $g=1$) can occur in the resolution of 
	$h_2(J_1,J_2,J_3,J_4)$.  We know that the coefficients of the ``minimal'' $g=0$ classes
	are given by the Moebius function of $NC(n).$  Thus
	\begin{equation}
		h_2(J_1,J_2,J_3,J_4) = \underbrace{2\CC_{(3,1)}+ \CC_{(2,2)}}_{g=0}+
		\underbrace{a_2((1,1,1,1))\CC_{(1,1,1,1)}}_{g=1}.
	\end{equation}
			
	It is tempting to hypothesize that the coefficients $a_r(\mu)$ have a topological 
	interpretation for arbitrary $g \geq 0,$ as in \cite{GJ} where it is found that
	the coefficients in the class basis resolution of transitive powers of JM elements
	are related to branched covers of the sphere.  This is supported by the fact that
	the $g=0$ case in our setting is determined by the Moebius function of the lattice
	of partitions whose graphical representations can be embedded on the sphere.
	
	\section{Character Expansion and Content Evaluation}
	
	\subsection{Character expansion via JM elements}
	Another basis of the center $Z(n)$ of the group algebra $\C[S(n)]$ is provided
	by the set $\{\chi^{\lambda}\}_{\lambda \vdash n}$ of irreducible characters of 
	$S(n).$  If $f \in Z(n)$ is a central function, its resolution
	\begin{equation}
		f=\sum_{\lambda \vdash n} f(\lambda)\chi^{\lambda}
	\end{equation}
	with respect to the character basis of $Z(n)$ is called the {\it character expansion}
	of $f.$  In this section we explicitly determine the character expansions of
	$\Wg_d$ and $h_r(J_1,\dots,J_n).$
	
	The character expansion of $\Wg_d$ can be deduced from Theorem
	\ref{main} and the following fundamental property of JM elements,
	which is due to Jucys.  Recall that the {\it content} $c(\Box)$ a cell $\Box$
	in a Young diagram $\lambda$ is the column index of $\Box$ subtract the
	row index of $\Box.$  
	For a Young diagram $\lambda,$ we denote by $A_{\lambda}$ the multiset 
	(or ``alphabet'') of
	its contents.  For example, if $\lambda=(4,2,1)$ then $A_{\lambda}=
	\{\{0,1,2,3,-1,0,-2\}\}.$
	For a symmetric polynomial $f$ in $n$ variables, $f(A_{\lambda})$ is the 
	{\it content evaluation} of $f$ at $\lambda.$  For our example diagram this
	would be
	\begin{equation}	
		f(0,1,2,3,-1,0,2).
	\end{equation}
	
	\begin{theorem}[\cite{J}]
		\label{content}
	For any symmetric polynomial $f$ in $n$ variables,
	\begin{equation}
		f(J_1,\dots,J_n)\chi^{\lambda}=f(A_{\lambda})\chi^{\lambda}.
	\end{equation}
	That is, $\chi^{\lambda}$ is an eigenvector for the linear operator
	``multiplication by $f(J_1,\dots,J_n)$'' on $Z(n)$ with corresponding
	eigenvalue $f(A_{\lambda}).$
	\end{theorem}
	
	Using Theorem \ref{content} we can find the character expansion
	of the Weingarten function.  
	Given a Young diagram $\lambda \vdash n,$
	let $H_{\lambda}$ denote the product of hook-lengths of $\lambda.$  
	Let 
		\begin{equation}
			s_{\lambda}(1^d) = \frac{1}{H_{\lambda}} \prod_{\Box \in \lambda}(d+c(\Box))
		\end{equation}
	be the Schur polynomial $s_{\lambda}$ evaluated at $(\underbrace{1,1\dots,1}_d)$
	(see \cite{S}).
	
	\begin{theorem}
		\label{characterExpansion}
	The character expansion of $\Wg_d \in \C[S(n)]$ is
	\begin{equation}
		\Wg_d=\sum_{\lambda \vdash n} \frac{1}{H_{\lambda}^2s_{\lambda}(1^d)}
		\chi^{\lambda}.
	\end{equation}
	\end{theorem}
	
	\begin{proof}
		By the second orthogonality
		relation for the irreducible characters of a finite group, the 
		character expansion of the unit $1$ of $Z(n)$ is
				\begin{equation*}
			1 = \frac{1}{n!} \sum_{\lambda \vdash n}\dim(\lambda) \chi^{\lambda},
		\end{equation*}
		where 
		\begin{equation}
			\dim(\lambda)=\frac{n!}{H_{\lambda}}
		\end{equation}
		is the number of standard Young tableaux of shape $\lambda.$
		Therefore by Corollary \ref{equivalence} we have the following:
		\begin{align*}
			\Wg_d &= \sum_{r \geq 0}\frac{a_r}{d^{n+r}} \\
			&= \sum_{r \geq 0}(-1)^rh_r(J_1,\dots,J_n)(\sum_{\lambda \vdash n} 			H_{\lambda}^{-1} \chi^{\lambda})\frac{1}{d^{n+r}} \\
			&=\sum_{\lambda \vdash n} H_{\lambda}^{-1}(\sum_{r \geq 0}(-1)^rh_r(A_{\lambda})
			\frac{1}{d^{n+r}})\chi^{\lambda}\\
			&= \sum_{\lambda \vdash n} H_{\lambda}^{-1} \prod_{\Box \in \lambda}
			(d+c(\Box))^{-1} \chi^{\lambda} \\
			&=\sum_{\lambda \vdash n} \frac{1}{H_{\lambda}^2s_{\lambda}(1^d)} 
			\chi^{\lambda}.
			\end{align*}
	\end{proof}
	
	Theorem \ref{characterExpansion} was derived in \cite{C} by a different method
	which does not involve JM elements.
	One benefit to our approach is that in the course of proving Theorem 
	\ref{characterExpansion}, we have shown that
	\begin{equation}
		a_r=(-1)^r \sum_{\lambda \vdash n} \frac{h_r(A_{\lambda})}{H_{\lambda}}
		\chi^{\lambda}.
	\end{equation}
	Thus we have the following.
	
	\begin{theorem}
	The character expansion of $h_r(J_1,\dots,J_n)$ is
	\begin{equation}
		h_r(J_1,\dots,J_n) = \sum_{\lambda \vdash n} \frac{h_r(A_{\lambda})}{H_{\lambda}}
		\chi^{\lambda}.
	\end{equation}
	\end{theorem}
	
	An interesting Corollary of this result is the formula
	\begin{equation}
		\Moeb(\mu) = (-1)^{n-\ell(\mu)}
		\sum_{\lambda \vdash n} \frac{h_{n-\ell(\mu)}(A_{\lambda})}{H_{\lambda}}
		\chi^{\lambda}(\mu)
	\end{equation}
	for the Moebius function of $NC(n).$  This in turn gives an esthetically pleasing
	family of identities relating content evaluation, hook-products, characters, and 
	Catalan numbers:
	\begin{equation}
		\sum_{\lambda \vdash n} \frac{h_{n-\ell(\mu)}(A_{\lambda})}{H_{\lambda}}
		\chi^{\lambda}(\mu)
		=\prod_{i=1}^{\ell(\mu)} \Cat_{\mu_i-1}.
	\end{equation}
	
	\section{Tables}
	Using Theorem \ref{characterExpansion}, one can easily compute the 
	asymptotic expansion of $\Wg_d \in Z(n),$ and thus the class basis resolution
	of $h_r(J_1,\dots,J_n)$ for all $r \geq 0,$ from the character table 
	of $S(n).$  In this appendix we tabulate this data for $n=2,3,4.$
	
	\subsection{$n=2$}
		
			\begin{itemize}
			
				\item
				$\Wg_d((1,1))=\frac{1}{d^2-1}$
				$$=\frac{1}{d^2} + \frac{1}{d^4} + \frac{1}{d^6} +
				\frac{1}{d^8}+\frac{1}{d^{10}}  + \frac{1}{d^{12}} \frac{1}{d^{14}}
				+\frac{1}{d^{16}}+\dots.$$

				\item
				$\Wg_d((2))=\frac{-1}{d(d^2-1)}$
				$$=-\frac{1}{d^3}-\frac{1}{d^5}-\frac{1}{d^7}-\frac{1}{d^9}
				-\frac{1}{d^{11}}-\frac{1}{d^{13}}-\frac{1}{d^{15}}-\dots$$
				
			\end{itemize}
			
			\begin{align*}
				h_0(J_1,J_2) &= \CC_{(1,1)} \\
				h_1(J_1,J_2) &= \CC_{(2)} \\
				&\vdots \\
				h_{2r}(J_1,J_2) &= \CC_{(1,1)} \\
				h_{2r+1}(J_1,J_2) &=  \CC_{(2)} \\
				&\vdots
			\end{align*}
			
		\subsection{$n=3$}
		
			\begin{itemize}
			
			\item
			$\Wg_d((1,1,1))=\frac{d^2-2}{d(d^2-1)(d^2-4)}$
			$$=\frac{1}{d^3}+\frac{3}{d^5}+\frac{11}{d^7}+\frac{43}{d^9}
			+\frac{171}{d^{11}}+\frac{683}{d^{13}}+\frac{2731}{d^{15}}+\dots$$
			
			\item
			$\Wg_d((2,1))=\frac{-1}{(d^2-1)(d^2-4)}$
			$$=-\frac{1}{d^4}-\frac{5}{d^6}-\frac{21}{d^8}-\frac{85}{d^{10}}
			-\frac{341}{d^{12}}-\frac{1365}{d^{14}}-\frac{5461}{d^{16}}-\dots$$
			
			\item
			$\Wg_d((3))=\frac{2}{d(d^2-1)(d^2-4)}$
			$$=\frac{2}{d^5}+\frac{10}{d^7}+\frac{42}{d^9}+\frac{170}{d^{11}}
			+\frac{682}{d^{13}}+\frac{2730}{d^{15}}+\dots$$
			
			\end{itemize}
			
			\begin{align*}
			h_0(J_1,J_2,J_3) &= \CC_{(1,1,1)} \\
			h_1(J_1,J_2,J_3) &= \CC_{(2,1)} \\
			h_2(J_1,J_2,J_3) &= 3 \CC_{(1,1,1)}+2\CC_{(3)} \\
			h_3(J_1,J_2,J_3) &= 5\CC_{(2,1)} \\
			h_4(J_1,J_2,J_3) &= 11\CC_{(1,1,1)}+10\CC_{(3)} \\
			h_5(J_1,J_2,J_3) &= 21\CC_{(2,1)} \\
			h_6(J_1,J_2,J_3) &= 43\CC_{(1,1,1)}+42\CC_{(3)} \\
			h_7(J_1,J_2,J_3) &= 85\CC_{(2,1)} \\
			h_8(J_1,J_2,J_3) &= 171\CC_{(1,1,1)} + 170\CC_{(3)} \\
			h_9(J_1,J_2,J_3) &= 341\CC_{(2,1)} \\
			h_{10}(J_1,J_2,J_3) &= 683\CC_{(1,1,1)}+682\CC_{(3)} \\
			h_{11}(J_1,J_2,J_3) &= 1365\CC_{(2,1)} \\
			&\vdots
			\end{align*}
			
		\subsection{$n=4$}
		
			\begin{itemize}
				
				\item
				$\Wg_d((1,1,1,1))=\frac{6-8d^2+d^4}{d^2
				(-36+49d^2-14d^4+d^6)}$
				$$=\frac{1}{d^4} +\frac{6}{d^6}+\frac{41}{d^8}+\frac{316}{d^{10}}
				+\frac{2631}{d^{12}}+\frac{22826}{d^{14}}+\frac{202021}{d^{16}}+\dots$$
				
				\item
				$\Wg_d((2,1,1))=\frac{1}{9d-10d^3+d^5}$
				$$=\frac{1}{d^5}+\frac{10}{d^7}+\frac{91}{d^9}+\frac{820}{d^{11}}
				+\frac{7381}{d^{13}}+\frac{66430}{d^{15}}+\frac{597871}{d^{17}}+\dots$$
				
				\item
				$\Wg_d((2,2))=\frac{6+d^2}{d^2(-36+49d^2-14d^4+d^6)}$
				$$=\frac{1}{d^6}+\frac{20}{d^8}+\frac{231}{d^{10}}+\frac{2290}{d^{12}}
				+\frac{21461}{d^{14}}+\frac{196560}{d^{16}}+\frac{1782691}{d^{18}}
				+\dots$$
				
				\item
				$\Wg_d((3,1))=\frac{-3+2d^2}{d^2(-36+49d^2-14d^4+d^6)}$
				$$=\frac{2}{d^6}+\frac{25}{d^8}+\frac{252}{d^{10}}+\frac{2375}{d^{12}}
				+\frac{21802}{d^{14}}+\frac{197925}{d^{16}}
				+\frac{1788152}{d^{18}}+\dots$$
				
				\item
				$\Wg_d((4))=\frac{5}{d(-36+49d^2-14d^4+d^6)}$
				$$=\frac{5}{d^7}+\frac{70}{d^9}+\frac{735}{d^{11}}
				+\frac{7040}{d^{13}}+\frac{65065}{d^{15}}+\frac{592410}{d^{17}}
				+\frac{5358995}{d^{19}}+\dots$$
				
			\end{itemize}
			
			\begin{align*}
				h_0(J_1,J_2,J_3,J_4)&=\CC_{(1,1,1,1)} \\
				h_1(J_1,J_2,J_3,J_4)&=\CC_{(2,1,1)} \\
				h_2(J_1,J_2,J_3,J_4)&=6\CC_{(1,1,1,1)}+
					\CC_{(2,2)}+2\CC_{(3,1)}  \\
				h_3(J_1,J_2,J_3,J_4)&=10\CC_{(2,1,1)}+5\CC_{(4)} \\
				h_4(J_1,J_2,J_3,J_4)&=41\CC_{(1,1,1,1)}
					+20\CC_{(2,2)} + 25\CC_{(3,1)} \\
				h_5(J_1,J_2,J_3,J_4)&=91\CC_{(2,1,1)}
					+70\CC_{(4)} \\
				h_6(J_1,J_2,J_3,J_4)&=316\CC_{(1,1,1,1)}
					+231\CC_{(2,2)}+252\CC_{(3,1)} \\
				h_7(J_1,J_2,J_3,J_4)&=820\CC_{(2,1,1)}+
				735\CC_{(4)} \\
				h_8(J_1,J_2,J_3,J_4)&=2631\CC_{(1,1,1,1)}+
					2290\CC_{(2,2)}+
					2375\CC_{(3,1)} \\
				h_9(J_1,J_2,J_3,J_4)&=7381\CC_{(2,1,1)}+
					7040\CC_{(4)} \\
				h_{10}(J_1,J_2,J_3,J_4)&=22826\CC_{(1,1,1,1)}+
					21461\CC_{(2,2)}+
					21802\CC_{(3,1)} \\
				h_{11}(J_1,J_2,J_3,J_4)&=66430\CC_{(2,1,1)}+
					65065\CC_{(4)} \\
				h_{12}(J_1,J_2,J_3,J_4)&=202021\CC_{(1,1,1,1)}+
					196560\CC_{(2,2)}+
					197925\CC_{(3,1)}.
			\end{align*}

 \end{document}